\documentclass{amsart}
\usepackage{amssymb}
\usepackage{graphics,color}
\usepackage[latin1]{inputenc}
\usepackage[pagebackref=false]{hyperref}

\newtheorem{theorem}{Theorem}[section]
\newtheorem{thm}[theorem]{Theorem}
\newtheorem{lem}[theorem]{Lemma}
\newtheorem{prop}[theorem]{Proposition}
\newtheorem{cor}[theorem]{Corollary}
\theoremstyle{definition}
\newtheorem{defn}[theorem]{Definition}
\theoremstyle{remark}
\newtheorem{rem}[theorem]{Remark}
\newtheorem{ex}[theorem]{Example}

\newcommand{\Max}{\mbox{Max}\,}

\renewcommand{\H}{\mbox{H}}

\newcommand{\N}{\mathbb{N}}

\newcommand{\fm}{\mathfrak{m}}

\newcommand{\AP}{\mbox{AP}}
\newcommand{\ord}{\mbox{ord}}

\begin{document}

\bibliographystyle{amsplain}

\title{On monomial curves obtained by gluing }

\author{Raheleh Jafari}
\address{Raheleh Jafari\\School of Mathematics, Institute for
Research in Fundamental Sciences (IPM) P. O. Box: 19395-5746,
Tehran, Iran.}

\email{rjafari@ipm.ir}

\author[S. Zarzuela Armengou]{Santiago Zarzuela Armengou}
\address{ Santiago Zarzuela Armengou\\Departament d'Àlgebra i Geometria, Universitat de Barcelona, Gran Via 585,
08007 Barcelona, Spain.}

\email{szarzuela@ub.edu}

\subjclass[2010]{13A30, 13H10, 13P10, 13D40}

\keywords{numerical semigroup ring, tangent cone, Hilbert function, semigroup gluing, Cohen-Macaulay. \\
Raheleh Jafari was in part supported by a grant from IPM (No. 900130068). \\
Santiago Zarzuela Armengou was supported by MTM2010-20279-C02-01.}

\maketitle

\begin{abstract}
We study arithmetic properties of tangent cones associated to affine
monomial curves, using the concept of gluing. In particular we
characterize the Cohen-Macaulay and Gorenstein properties of tangent
cones of some families  of monomial curves obtained by gluing. Moreover,
we provide new families of monomial curves with non--decreasing Hilbert functions.
\end{abstract}

\section{introduction}
A monomial curve $C$ in the affine space $\mathbb{A}^d_k$ over a
field $k$ consists on the set of points defined parametrically by
$X_{1}=t^{m_1}, \ldots, X_{d}=t^{m_d}$, for some positive integers
$m_1<\cdots<m_d$. In order to be sure that different
parameterizations give rise to different monomial curves, we may
assume that $\gcd(m_1,\ldots,m_d)=1$. In fact, it is known that the
set $C$ is an affine variety whose coordinate ring is
$\mathcal{R}=k[t^{m_1},\ldots,t^{m_d}]$, see for instance E. Reyes, R. H.
Villarreal and L. Zárate \cite{RVZ}. The set
$S=\{r_1m_1+\cdots+r_dm_d ; r_i\geq 0\}$ is a subset of the
non--negative integers $\N\cup\{0\}$ which is closed under addition,
and the condition $\gcd(m_1,\ldots,m_d)=1$ is equivalent to the
property $\#\,N\setminus S<\infty$. In other words,
$S=<m_1,\ldots,m_d>$ is a numerical semigroup minimally generated by
the unique minimal system of generators $\{m_1,\ldots,m_d\}$. The
coordinate ring $\mathcal{R}$ is called the  numerical semigroup
ring associated to $S$ and it is denoted by $k[S]$. Since we are
interested in the arithmetical properties at the origin, which is
the only singular point of the curve $C$, we will consider the ring
$R=k[[t^{m_1},\ldots,t^{m_d}]] = k[[S]]$. Note that $R$ is a
complete one--dimensional local domain with maximal ideal
$\fm=(t^{m_1},\ldots,t^{m_d})$. We also consider the tangent cone
associated to $k[[S]]$; that is the graded ring
 \[ G(S):=\bigoplus_{n\geq 0} \fm^n/\fm^{n+1}. \]

  We say that a numerical semigroup $S$ is complete intersection if the ring $k[[S]]$ is complete intersection (and similarly for other properties like Gorenstein, Cohen-Macaulay or Buchsbaum).
  J.~C.~Rosales introduced the concept of gluing in his PhD thesis \cite{R0} in order to characterize complete intersection numerical semigroups. In particular, he proved that a numerical semigroup other than $\N$ is complete intersection if and only if
  it is a gluing of two complete intersection numerical semigroups. The definition of gluing can also be
  generalized to affine semigroups, as finitely generated submonoids of $\N^n$, see \cite{R}. By E. Kunz \cite{K} it is also well known that the Gorenstein property of a numerical semigroup $S$ is equivalent to the symmetric property of $S$. This property is in fact preserved by gluing, in the sense that the gluing of two symmetric numerical semigroups is symmetric, see P. A. García-Sánchez and J. C. Rosales \cite[Proposition 8.11]{GR}. Hence the concept of gluing is an effective tool to construct families of complete intersection numerical semigroups and also families of Gorenstein numerical semigroups which are not complete intersection, and in this way several authors have studied and used gluing techniques, as for instance F. Arslan and P. Mete \cite {AM}, F. Arslan, P. Mete and M. \c{S}ahin \cite{AMS}, M. Morales and A. Thoma \cite{Mo}, M. \c{S}ahin \cite{Sah}, or A. Thoma \cite{T}.

 On the contrary, the study of the tangent cones of numerical semigroups obtained by gluing has not been so carefully analyzed. As a natural approach we may ask which properties of the tangent cones are preserved under gluing. At first instance the answer is not so hopeful: Arslan-Mete-\c{S}ahin  give in \cite{AMS} an example of two monomial curves whose tangent cones are Cohen-Macaulay
 but their gluing has a non-Cohen-Macaulay tangent cone. Then, they introduce a special type of
gluing, called nice gluing, that in fact preserves the
Cohen-Macaulay property of tangent cones \cite[Theorem 2.6]{AMS}. On the other hand, it
is well known that if the tangent cone $G(S)$ is Cohen-Macaulay then
the Hilbert function of $k[[S]]$ (that we also call the Hilbert
function of $S$) is non--decreasing, thus the previous result has
the nice consequence that the Hilbert function of a numerical
semigroup obtained by the nice gluing of two Cohen-Macaulay
numerical semigroups is non--decreasing. In the same paper, they
also generalize this result showing that if a numerical semigroup
$S$ is a nice gluing of two numerical semigroups $S_1$ and $S_2$
such that $S_1$ has a non--decreasing Hilbert function and $G(S_2)$
is Cohen-Macaulay, then $S$ has a non--decreasing Hilbert function,
\cite[Theorem 3.1]{AMS}. This result allows to provide large
families of Gorenstein monomial curves with non--Cohen-Macaulay
tangent cones having non-decreasing Hilbert functions, giving
support to a conjecture due to M. E. Rossi saying that every
$1$-dimensional Gorenstein local ring  has a non--decreasing Hilbert
function.

In this paper, we introduce a new kind of gluing, that we call specific
gluing, which contains a large class of glued numerical semigroups obtained from two given numerical semigroups
$S_1$ and $S_2$. In particular it
contains all the numerical semigroups $S$ obtained by
a nice gluing of two arbitrary numerical semigroups $S_1$ and $S_2$ such that $G(S_2)$ is Cohen-Macaulay.
For instance, one of our results says that if $S$ is a specific gluing of two numerical semigroups $S_1$ and $S_2$,
then $G(S)$ is Cohen-Macaulay if and only if $G(S_1)$ is Cohen-Macaulay, see Theorem \ref{nc}. In particular, a nice gluing of two numerical
 semigroups $S_1$ and $S_2$ such that $G(S_1)$ and $G(S_2)$ are Cohen-Macaulay has a Cohen-Macaulay tangent cone, which provides a different simple proof for the already cited \cite[Theorem 2.6]{AMS}. Our approach enables us to prove a similar result for the Gorenstein property of the tangent cone in Theorem \ref{Gor}.
We also generalize \cite[Theorem 3.1]{AMS}, by showing that the Hilbert function of a specific gluing $S$ of two numerical semigroups $S_1$  and $S_2$ is non--decreasing, provided that $S_1$ has a non--decreasing Hilbert function. As a special case of gluing, we study more in detail extensions and state that for an integer $q>1$, all the extensions of a numerical semigroup by $q$, except finitely many of them, have non--decreasing Hilbert functions. As a consequence we may construct a large family of free monomial curves, with non--decreasing Hilbert functions.

A quite usual method to study the tangent cone of a local ring
consists on finding a presentation of it in terms of the ideal of
initial forms of a defining ideal of the ring. This is not easy, but
sometimes it can be done by an explicit computation of standard
basis, as for instance it was done in L. Robbiano and G. Valla
\cite{RV}. For the case of monomial curves, F. Arslan used in
\cite{A} a Gröbner basis strategy in order to compute this ideal of
initial forms. On the other hand, gluing is a technique somehow well
behaved to deal with presentations, in the sense that given a
numerical semigroup $S$ which is a gluing of two numerical
semigroups $S_1$ and $S_2$, one has a control of a presentation of
$k[[S]]$ in terms of the presentations of $k[[S_1]]$ and $k[[S_2]]$.
The combination of these two strategies is basically the technique
used by Arslan-Mete-\c{S}ahin in \cite{AMS}. On the contrary, our
approach is completely different and based on Apéry sets, something
easier to compute and more intrinsically related to the semigroup
itself. This approach has been used several times in studying
tangent cones of monomial curves. In fact, we have already used this
method in a previous work jointly with T. Cortadellas, see T.
Cortadellas, R. Jafari, and S. Zarzuela \cite{CJZ}, which is based
on the study of the Apéry tables of monomial curves introduced in T.
Cortadellas and S. Zarzuela \cite{CZ1}. Also, M. D'Anna, V. Micale
and A. Sammartano \cite{DMS} have recently used Apéry sets to
characterize the complete intersection property of the tangent cone
of a monomial curve. And we have used some of the ideas involved in
the paper by L. Bryant \cite{B} as well.

For the general concepts of Commutative Algebra used in this paper one may consult the book of W. Bruns and J. Herzog \cite{BH}, while for the general notations and results on numerical semigroups and numerical semigroup rings we shall use the books by P. A. García Sánchez and J. C. Rosales \cite{GR} and by V. Barucci, D. E. Dobbs, and M. Fontana \cite{BDF}. Some of the explicit examples along  this paper have been computed by using the NumericalSgps package of GAP \cite{Num}.

This work was initiated during a two months stay in the year 2012 of the first author at the Institute of Mathematics of the University of Barcelona (IMUB), in the frame of its program of visiting researchers. Both authors would like to thank the IMUB for its hospitality and support. Finally, we also would like to thank Tere Cortadellas for a careful reading of this manuscript.

\section{Preliminaries}
Throughout $S=<m_1,\ldots,m_d>$ will be a numerical semigroup
minimally generated by $m_1<\cdots<m_d$ and $M=S\setminus\{0\}$ will
denote the maximal ideal of $S$. Recall that a relative ideal of $S$
is a nonempty set $H$ of integers such that $H+S\subseteq H$ and
$s+H\subseteq S$ for some $s\in S$. A relative ideal of $S$ is
called an ideal if it is contained in $S$.  Note that if $H_1$ and
$H_2$ are relative ideals of $S$, then $H_1+H_2=\{h_1+h_2; h_1\in
H_1 , h_2\in H_2\}$ is also a relative ideal of $S$. In particular,
for $z\in \mathbb{Z}$, $z+S=\{z+s; s\in S\}$ is the principal
relative ideal of $S$ generated by $z$. Let $\nu$ denote the
$t$-adic valuation in $k[[t]]$. If $I$ is a fractional ideal of
$k[[S]]$ then $\nu(I)$ is a relative ideal of $S$ and for fractional
ideals $J\subset I$, we have $\lambda(I/J)=\#\nu(I)\setminus
\nu(J)$. Hence $\nu(k[[S]])=S$ and
$\nu(\fm^n)=nM=M+\overset{n}{\cdots}+M$ for $n\geq 1$.

The element $t^{m_1}$, where $m_1$ is the smallest element in the
set of generators of $S$, generates a minimal reduction of
$\fm$ and $m_1$ is equal to the multiplicity of $R = k[[S]]$. The
element $m_1$ is called the multiplicity of $S$ and it is denoted by
$m(S)$.

For an element $s\in S$, the maximum  integer $n$ such that $nM$
contains  $s$, is called the order of $s$ and it is denoted by
$\ord_S(s)$. In other words, $s\in nM\setminus (n+1)M$ if and only
if $n=\ord_S(s)$. So that  $s$ may be written as
$s=\sum^d_{i=1}r_im_i$ $(r_i\geq 0)$, such that $\sum^d_{i=1}r_i=n$.
We call this representation a maximal expression of $s$. Because $S$ is minimally generated by $m_1<\cdots<m_d$ we have obviously
that $\ord_S(m_i)=1$ for all $i=1, \dots, d$.

We will denote by $\AP(S,s)$ the Apéry set of $S$ with respect to $s$; that
is the set of the smallest elements in $S$ in each congruence class
modulo $s$, equivalently, the set of elements $x\in S$ such that
$x-s\notin S$. In particular $S=\AP(S,s)+rs$, $r\in\N$.
 We consider a natural partial ordering
$\preceq$ on $S$ where, for all elements $x$ and $y$ in $S$,
$x\preceq y$ if there is an element $z$ in $S$ such that $y=x+z$. Note that if $y\in \AP(S,s)$ and $x\in S$ is such that $x\preceq y$, then $x\in \AP(S,s)$. We use also another partial ordering denoted by $\preceq_M$, where
$x\preceq_M y$ if $y= x+z$ and $\ord(y)=\ord(x)+\ord(z)$ for some
$z\in S$. Considering these orderings, the maximal elements of
$\AP(S,n)$ are denoted respectively by $\Max\AP(S,s)$ and
$\Max_M\AP(S,s)$. It is clear that $\Max\AP(S,s) \subseteq \Max_M\AP(S,s)$.

For an element $s$ of order $n$, we denote the initial form of
$t^s\in\fm^n\setminus\fm^{n+1}$ by
$(t^s)^*\in\fm^n/\fm^{n+1}\hookrightarrow G(S)$. Indeed we consider
 the map
\begin{equation}
\begin{array}{lll}
S & \longrightarrow & G(S)\\
s & \mapsto & (t^s)^{\ast}.
\end{array}
\end{equation}

Note that   $(t^s)^*(t^{s'})^*=0$ for two elements $s,s'\in S$, if
and only if $\ord_S(s+s')>\ord_S(s)+\ord_S(s')$. This fact allows to state
the following useful lemma.

\begin{lem}\label{zero}
The following statements hold for $e=m(S)$.
\begin{enumerate}
\item $(t^{ae})^*\neq 0$ for all $a>0$. In particular $\ord_S(ae)=a$
for all $a>0$.
\item If $(t^x)^*$ is a non--zero--divisor over the set of elements of the form $(t^{jx})^{\ast} \in G(S)$ with $1 \leq j \leq e-1$, then $x=ae$ for some $a\in\N$.
\end{enumerate}
\end{lem}

\begin{proof}
If $e=1$, then $S=\N$ and both statements are clear. So we may assume
that $e>1$.

(1) If $(t^{ae})^*=0$, then $\ord_S(ae)>a$. Thus
$ae=\sum^d_{i=1}r_im_i$ for some integers $r_i\geq 0$ such that
$\sum^d_{i=1}r_i>a$. Note that $e=m_1<m_i$ for $2\leq i\leq d$, so
that $ae=\sum^d_{i=1}r_im_i\geq (\sum^d_{i=1}r_i)e>ae$ which is a
contradiction.

(2) By hypothesis, $(t^{x})^*(t^{(e-i)x})^*\neq0$ for all $1\leq i\leq e-1$. Taking $i=e-1$ we get that
$\ord_S(2x)=2\ord_S(x)$ and repeating this process recursively we get that
$\ord_S(ex)=e\ord_S(x)$. On the other hand $\ord_S({ex})=x$ by part (1). Hence $x=\ord_S(x)e$.
\end{proof}

Because $t^{m_1}$ is a minimal reduction of $\fm$, its initial form $(t^{m_1})^*$ is a homogeneous parameter of the graded ring $G(S)$. Hence $G(S)$ is Cohen-Macaulay if and only if $(t^{m_1})^*$ is a non-zero divisor. The following lemma will be very useful in order to detect the Cohen-Macaulay property of the tangent cone: the implications $(1) \Rightarrow (2) \Rightarrow (3) \Rightarrow (4) \Rightarrow (5)$ are direct consequences of the previous considerations, while the implication $(5) \Rightarrow (1)$ is a well known result proved for the first time by A. García in \cite[Theorem 7, Remark 8]{G}, see also \cite[Remark 2.11 ]{CJZ} for a different proof.

\begin{lem}\label{CM}
The following statements are equivalent for $e=m(S)$.
\begin{enumerate}
\item $G(S)$ is Cohen-Macaulay.
\item $(t^e)^*$ is a non--zero--divisor of $G(S)$.
\item $(t^e)^*$ is a non--zero--divisor over the set of elements of the form $(t^s)^{\ast} \in G(S)$.
\item $\emph\ord_S(s+e)=\emph\ord_S(s)+1$ for all $s\in S$.
\item $\emph\ord_S(w+ae)=\emph\ord_S(w)+a$ for all $w\in\emph\AP(S,e)$ and $a\geq
0$.
\end{enumerate}
\end{lem}

Let $e=m(S)$ and set $r$ the reduction number of $S$, that is,
$$r=\min\{r\in\N ; \fm^{r+1}=t^e\fm^r\}=\min\{r ; (r+1)M=e+rM \}$$
Observe that by the definition of Apéry set with respect to $e$, the
reduction number of $S$ is at most $e-1$. It's clear that
$\ord_S(s+e)=\ord_S(s)+1$ for all $s\in S$ with $\ord_S(s)\geq r$,
hence in order to check the Cohen-Macaulay property of $G(S)$ it is
enough to check the condition $(3)$ in the above lemma only for
elements $s\in S$ with $\ord_S(s)\leq r$. This motivates the
introduction of the following definition: for any $x\in S$ let
\begin{equation}
l_x(S):=\max\{\ord_S(s+x)-\ord_S(x)-\ord_S(s) ; s\in S, \ord(s)\leq
r\}.
\end{equation}

We then have the following characterization.

\begin{cor}\label{CMl}
The tangent cone of a numerical semigroup $S$ is Cohen-Macaulay if
and only if $l_e(S)=0$, where $e=m(S)$.
\end{cor}

\begin{rem}\label{lr}
Using the notation of Apéry table introduced in \cite{CZ1}, we
observe that $l_e(S)$ is indeed the maximum length of true landings
in the Apéry table of $S$. In particular $l_e(S)\leq r-1$.
\end{rem}

\begin{rem}\label{L}
Let $s\in S$ and assume that $\ord_S(s) > r$. Then $s=y+ke$ for some
$y\in S$ with  $\ord_S(y)=r$. Now, for an element $x\in S$, we have
$\ord_S(s+x)=\ord_S(y+ke+x)$ and since $\ord_S(y+x)\geq r$, we
obtain that
$$\ord_S(s+x)-\ord_S(x)-\ord_S(s)=\ord_S(x+y)+k-\ord_S(x)-\ord_S(y)-k\leq
l_x(S).$$
On the other hand, if $\ord_S(s)\leq r$ we have by definition that
$$\ord_S(s+x) -\ord_S(x)-\ord_S(s) \leq l_x(S).$$
Hence for elements $s_1,s_2\in S$, we have
\begin{equation}
\ord_S(s_1+s_2)\leq \ord_S(s_1)+\ord_S(s_2)+\min\{l_{s_1}(S) ,
l_{s_2}(S)\} .\end{equation}

Note also that to bound $l_x(S)$ we only need to consider the values of $l_s(S)$ for elements $s$
of $S$ with $\ord_S(s)\leq r$, that is
$$l_x(S) \leq \max \{ l_s(S); \ord_S(s)\leq r\}.$$
Hence it is natural to consider the finite number
\begin{equation}
L(S)=\max\{l_s(S) ; \ord_S(s)\leq r\}. \end{equation} Then
$\ord_S(s_1+s_2)\leq \ord_S(s_1)+\ord_S(s_2)+L(S)$ for all $s_1,
s_2\in S$.
\end{rem}

\begin{rem}\label{beta}
Assume that $x\in S$ and consider the zero-dimensional ring $\bar{R}=R/t^xR$. Then $\bar{R}$
is of length $\lambda(\bar{R})=\#\AP(S,x)$. Set $\bar{G}(S)$ the tangent cone of $\bar{R}$:
$\bar{G}(S)=\bigoplus_{i\geq 0}\bar{G}_i$, where
$\bar{G}_i=\frac{(\fm/t^x)^i}{(\fm/t^x)^{i+1}}$ and $\lambda(\bar{G}_i)=\#\{w\in\AP(S,x) ; \ord(w)=i\}$. In particular
$\bar{G}_i=0$ for all $i>\max\{\ord_S(w); w\in\AP(S,x)\}$. We will
use the following notations to refer to these numbers.
\begin{equation}
\beta_i(x):=\lambda(\bar{G}_i)=\#\{w\in\AP(S,x) ; \ord(w)=i\}
\end{equation}
\begin{equation}
d(x):=\max\{i\in\N; \bar{G}_i\neq 0\}=\max\{\ord_S(w);
w\in\AP(S,x)\}
\end{equation}
\end{rem}

According to the above remark, many properties of $G(S)$ which are
induced by the quotient ring $\bar{R}$ or the corresponding tangent
cone $\bar{G}(S)$, could be detected in terms of Apéry sets. On the
other hand, it is well known that the numerical semigroup ring $R$
is Gorenstein if and only if $S$ is symmetric (see the
introduction), which in terms of Apéry sets is equivalent to that
given a (any) positive integer $n\in S$, if
$\AP(S,n)=\{0<a_1<\cdots<a_{n-1}\}$ then $a_i+a_{n-i-1}=a_{n-1}$ for
all $i=0,\ldots,n-1$. This is also equivalent to say that
$\Max\AP(S,n)$ has only one element. Note that if $S$ is symmetric
then $\bar{R}=R/t^xR$ is also Gorenstein for any $x\in S$.

A numerical semigroup $S$ is called pure (resp. $M$--pure) if all elements in
$\Max \AP(S,e)$ (resp. $\Max_M\AP(S,e)$) have the same order, where $e=m(S)$ is the
multiplicity of $S$. It may be seen \cite[Proposition 3.4]{B} that $S$ is $M$-pure if and only if $S$ is pure and $\Max \AP(S,e) = \Max_M\AP(S,e)$. In particular $S$ is $M$--pure symmetric if and only if $\Max_M\AP(S,e)$ has only one element. By a result of  L. Bryant \cite[Theorem
3.14]{B}, we know that $G(S)$ is Gorenstein if and only if $G(S)$ is
Cohen-Macaulay and $S$ is $M$-pure symmetric.

The above notions can be naturally extended to the relative case: given a numerical semigroup $S$ and $x\in S$, we say that $S$ is pure (resp. $M$-pure) with respect to $x$ if all elements in $\Max \AP(S,x)$ (resp. $\Max_M\AP(S,x)$) have the same order. Then it is easy to see that $S$ is $M$-pure with respect to $x$ if and only if $S$ is pure with respect to $x$ and $\Max \AP(S,x) = \Max_M\AP(S,x)$. In particular $S$ is symmetric and $M$--pure with respect to $x$ if and only if $\Max_M\AP(S,x)$ has only one element.

The following lemma gives a criteria for the Gorenstein property of the tangent cone of
$R/t^xR$ for an arbitrary element $x\in S$.  It extends the result proved
by Bryant in \cite{B} when $x=e$.

\begin{lem}\label{MG}

Assume that $S$ is symmetric and $x\in S$. Let $\bar{R}=R/t^xR$. The following statements
are equivalent.
\begin{enumerate}
\item $\bar{G}(S)$ is Gorenstein
\item $\beta_i(x)=\beta_{d(x)-i}(x)$ for $0\leq i\leq
\lfloor\frac{d(x)+1}{2}\rfloor$.
\item $S$ is $M$-pure with respect to $x$.
\end{enumerate}
 \end{lem}
\begin{proof}
First we show that (1) and (2) are equivalent. Let
$\bar{G}(S)=\oplus_{i=0}^{d(x)}\bar{G}_i$, where
$\bar{G}_i=\frac{(\fm/t^x)^i}{(\fm/t^x)^{i+1}}$. By \cite[Theorem 1.5]{O} (see also \cite[Theorem
3.1]{HKU}), $\bar{G}(S)$ is Gorenstein if and only if
$\lambda(\bar{G}_i)=\lambda(\bar{G}_{d(x)-i})$ for $0\leq i\leq
\lfloor\frac{d(x)+1}{2}\rfloor$. Now the result follows by Remark
\ref{beta}.

Now, we prove the equivalency of (2) and (3) by using the ideas
given in the proof of \cite[Proposition 3.8]{B}. Note that since $S$ is symmetric by hypothesis, $S$ is $M$-pure with respect to $x$ if and only if $\Max_M\AP(S,x)$ has only one element. Moreover, $\Max\AP(S,x)$ has a unique element say $w$ and so $\#\Max_M\AP(S,x)=1$ if and only if
$\Max_M\AP(S,x)=\Max\AP(S,x)=\{w\}$.

Assume that (3) holds and let $s\in\AP(S,x)$. Then $s\preccurlyeq_M w$,
that is  $s+y=w$ for some $y\in\AP(S,x)$ and
$\ord(s)+\ord(y)=\ord(w)=d(x)$. Hence the number of elements of
order $i$ in $\AP(S,x)$ is equal to the number of elements of order
$d(x)-i$, and so (2) follows.

Now assume that $\beta_i(x)=\beta_{d(x)-i}(x)$ for $0\leq i\leq
\lfloor\frac{d(x)+1}{2}\rfloor$. Let $$A_i=\{s\in\AP(S,x) ;
\ord(s)\geq d(x)-i\}$$ and $$B_i=\{w-s ; s\in\AP(S,x) \mbox{ and
} \ord(s)\leq i\}.$$ If $y\in A_i$, then $y\in\AP(S,x)$ and
$\ord(y)\geq d(x)-i$. On the other hand $y+z=w$ for some
$z\in\AP(S,x)$. Now,
 $d(x)=\ord(w)\geq\ord(y)+\ord(z)\geq d(x)-i+\ord(z)$. Hence $\ord(w-y)=\ord(z)\leq i$ and so $y\in B_i$. Note that
 $\#A_i=\sum_{j=0}^i\beta_{d(x)-j}=\sum_{j=0}^i\beta_j=\#B_i$. Thus $A_i=B_i$. Let $s\in\AP(S,x)$: then $s+n=w$ for some
 $n\in\AP(S,x)$. Note that
 $d(x)=\ord(w)\geq \ord(s)+\ord(n)$ and so
 $n\in B_{\ord(s)}=A_{\ord(s)}$. Thus $\ord(n)\geq d(x)-\ord(s)$.
 Hence $\ord(s)+\ord(n)=\ord(w)$, that is $s\preccurlyeq_Mw$. So that
 $\Max_M\AP(S,x)=\{w\}$ and we are done.
\end{proof}

Now, we are able to give an equivalent condition for the $M$--pure
symmetric property which will play an important role in our approach to study this property for monomial curves obtained by gluing. It gives some flexibility in order to test the Gorenstein property of the tangent cone.

\begin{prop}\label{PS}
Let $S$ be a numerical semigroup with multiplicity $e$. If
$G(S)$ is Cohen-Macaulay, then $G(S)$ is Gorenstein if and only if $S$ is symmetric and any of the following equivalent statements holds:

\begin{enumerate}
\item $S$ is $M$-pure.
\item $\beta_i(ke)=\beta_{d(ke)-i}$ for some $k>0$ and  $0\leq i\leq
\lfloor\frac{d(ke)+1}{2}\rfloor$.
\item $\beta_i(ke)=\beta_{d(ke)-i}$ for all  $k>0$  and $0\leq i\leq
\lfloor\frac{d(ke)+1}{2}\rfloor$.
\item 
$S$ is $M$-pure with respect to $ke$ for all $k>0$.
\item 
$S$ is $M$-pure with respect to $ke$ for some $k>0$.
\end{enumerate}
\end{prop}

\begin{proof}
First note that since $G(S)$ is Cohen-Macaulay, then $(t^{ke})^* \in G(S)$ is a non-zero-divisor for any $k>0$. In particular, $G(S)$ is Gorenstein if and only if $\bar{G}(S)$ is Gorenstein, for $\bar{R} = R/t^{e}R$. Hence by \cite[Proposition 3.14(1)]{B} $G(S)$ is Gorenstein if and only if $S$ is $M$-pure symmetric. Now, it suffices to observe that if $S$ is symmetric, conditions (1) to (5) are equivalent.

(1) $\Rightarrow$ (2) follows by \cite[Proposition 3.8]{B},
considering $k=1$.

(2) $\Rightarrow$ (3)  By Lemma \ref{MG}, $\bar{G}(S)$ is Gorenstein, for $\bar{R} = R/t^{ke}R$. Hence $G(S)$ is also Gorenstein and so $\bar{G}(S)$ is Gorenstein, for $\bar{R} = R/t^{ke}R$ and for all $k>0$. Again by Lemma \ref{MG} we get (3).

(3) $\Rightarrow$ (4) is obvious by Lemma \ref{MG}.

(4) $\Rightarrow$ (5) follows by a similar argument as in (2) $\Rightarrow$ (3).

(5) $\Rightarrow$ (1) By definition for $k=1$.

\end{proof}

\section{Tangent Cones of monomial curves obtained by gluing}

Throughout this section  $S_1$ and $S_2$  are two numerical semigroups
minimally generated by $m_1<\cdots<m_d$
 and $n_1<\cdots<n_k$ respectively.  First we recall the concepts of gluing and nice gluing of two numerical semigroups
as defined in \cite{R} and \cite{AMS} and respectively.

\begin{defn}\label{G}   Let $p\in S_1$ and $q\in S_2$ be two positive integers
 satisfying $\gcd(p,q)=1$ with $p\notin\{m_1,\ldots,m_d\}$ and
 $q\notin\{n_1,\ldots,n_k\}$. The numerical semigroup
 $$S=<qm_1,\ldots,qm_d,pn_1,\ldots,pn_k>$$ is called a \emph{gluing} of $S_1$ and $S_2$. We call $S$ a nice gluing of $S_1$ and
 $S_2$ if $q=an_1$ for some $1<a\leq \ord_{S_1}(p)$.
\end{defn}

\begin{rem}\label{m}
It is easy to see from the definition that the set
$$\{qm_1,\ldots,qm_d,pn_1,\ldots,pn_k\}$$ is a minimal system of
generators of $S$ and $m(S)=\min\{qm_1,pn_1\}$. If $S$ is a nice
gluing of $S_1$ and $S_2$, then $qm_1=an_1m_1\leq
\ord_{S_1}(p)n_1m_1\leq pn_1$ and so in this case, $m(S)=qm_1$.
\end{rem}

The following lemma states an easy but useful property about the representations of the elements of a semigroup obtained by
gluing.

\begin{lem}\label{ex}
Let $S=<qm_1,\ldots,qm_d,pn_1,\ldots,pn_k>$ be a gluing of $S_1$ and
$S_2$. Let $u=qx+py\in S$, where $x\in S_1$ and $y\in S_2$. Then
\begin{enumerate}
\item $\emph\ord_S(u)\geq \emph\ord_{S_1}(x)+\emph\ord_{S_2}(y)$.
\item There exist elements $z_1\in S_1$ and $z_2\in S_2$ such that
$u=qz_1+pz_2$ and $\emph\ord_S(u)=\emph\ord_{S_1}(z_1)+\emph\ord_{S_2}(z_2)$.
\end{enumerate}
\end{lem}
\begin{proof}
(1)  Let  $d_1=\ord_{S_1}(x)$ and $d_2=\ord_{S_2}(y)$. Consider
maximal expressions $x=\sum_{i=1}^dr_im_i$ and
$y=\sum_{j=1}^ks_jn_j$ with $\sum_{i=1}^dr_i=d_1$ and
$\sum_{j=1}^ks_j=d_2$. Hence we may write
$u=\sum^d_{i=1}r_iqm_i+\sum_{j=1}^ks_jpn_j$ and so $\ord_S(u)\geq
d_1+d_2$.

(2) Let $u=\sum^{d}_{i=1}r_iqm_i+\sum^k_{j=1}s_jpn_j$ be a maximal
expression of $u$, that is $ord_S(u)=\sum^d_{i=1}r_i+\sum^k_{j=1}s_j$. Note
that $z_1=\sum^d_{i=1}r_im_i$ and $z_2=\sum^k_{j=1}s_jn_j$ are also
maximal expressions. So that $u=qz_1+pz_2$ with
$\ord_S(u)=\ord_{S_1}(z_1)+\ord_{S_2}(z_2)$.
\end{proof}

In studying the Cohen-Macaulay property of the tangent cone of a numerical semigroup obtained by gluing the following simple characterization of a nice gluing will be a crucial point.

\begin{prop}\label{nice}
The following statements are equivalent for a numerical semigroup
$S=<qm_1,\ldots,qm_d,pn_1,\ldots,pn_k>$ which is a gluing of $S_1$
and $S_2$.
\begin{enumerate}
\item  $S$ is a nice gluing of $S_1$ and $S_2$, and $G(S_2)$ is
Cohen-Macaulay.
\item  $\emph\ord_{S_2}(q)\leq \emph\ord_{S_1}(p)$ and $(t^q)^*$ is a non--zero--divisor of $G(S_2)$.
\item $\emph\ord_{S_2}(q)\leq \emph\ord_{S_1}(p)$ and
$\emph\ord_{S_2}(y+\alpha q)=\emph\ord(y)+\alpha\emph\ord(q)$ for all $y\in S_2$
and $\alpha\geq 0$.
\end{enumerate}
\end{prop}

\begin{proof}

(1)$\Rightarrow$(2) By Definition \ref{G}, $\ord_{S_2}(q)\leq \ord_{S_1}(p)$ and $q=an_1$ for some $a>1$. Since $G(S_2)$ is
Cohen-Macaulay, $(t^{n_1})^*$ is a non-zero-divisor of $G(S_2)$ and so $(t^q)^*$ as well.

(2)$\Rightarrow$(3) is obvious, because if $(t^q)^*$ is a non-zero-divisor of $G(S_2)$, $(t^{\alpha q})*$ is also a non-zero-divisor of $G(S_2)$ for any $\alpha \geq 0$, and so it holds the condition on the orders in (3).

(3)$\Rightarrow$ (1) By taking $\alpha =1$, the condition on the orders in (3) means that $(t^q)^*$ is non-zero-divisor over the set of elements of the form $(t^s)^* \in G(S_2)$. Hence by Lemma \ref{zero}, $q=an_1$ for some $a>1$ and so $(t^{n_1})^*$ is also a non-zero-divisor over the set of elements of the form $(t^s)^* \in G(S_2)$. By Lemma \ref{CM}, $G(S_2)$ is then Cohen-Macaulay. On the other hand, since $\ord_{S_2}(q)=a$ we have that $a\leq\ord_{S_1}(p)$ and so $S$ is a nice gluing of $S_1$ and $S_2$.

\end{proof}

Observe that if $G(S_2)$ is Cohen-Macaulay then $l_q(S_2) = 0$ and
so condition (1) in the above proposition says in particular that
$q=an_1$ and $\ord_{S_2}(q) + l_q(S_2) \leq \ord_{S_1}(p)$. As a
general case of  this property we single out the following
definition.

\begin{defn}\label{sg}
Let $S=<qm_1,\ldots,qm_d,pn_1,\ldots,pn_k>$ be a gluing of $S_1$ and
$S_2$. We call $S$ a \emph{specific gluing} of $S_1$ and $S_2$,  if $\ord_{S_2}(q)+l_q(S_2)\leq \ord_{S_1}(p)$.
\end{defn}

\begin{rem}\label{fs}
Let $S_1$ and $S_2$ be numerical semigroups. By the definition it is clear that for a given $q$, all
possible  gluings  of $S_1$ and $S_2$ with $q$ are specific, except finitely
many of them.
\end{rem}

\begin{rem}\label{ns}
Let  $S=<qm_1,\ldots,qm_d,pn_1,\ldots,pn_k>$ be a  gluing  of $S_1$
and $S_2$.  If $S$ is a nice gluing of $S_1$ and $S_2$ and $G(S_2)$ is Cohen-Macaulay, then $S$ is a
specific gluing of $S_1$ and $S_2$.
\end{rem}

Now, we describe some Apéry sets of a numerical semigroup $S$
obtained by gluing of $S_1$ and $S_2$, in terms of Apéry sets of
$S_1$ and $S_2$.

\begin{prop}\label{AP}
Let $S=<qm_1,\ldots,qm_d,pn_1,\ldots,pn_k>$ be a gluing of $S_1$ and
$S_2$. The following statements hold for all $x\in S_1$.
\begin{enumerate}
\item $\emph\AP(S,qx)=\{qz_1+pz_2 ; z_1\in\emph\AP(S_1,x), z_2\in\emph\AP(S_2,q)\}$.
\item If $qz_1+pz_2\in\emph\AP(S,qx)$, then $z_1\in\emph\AP(S_1,x)$.
\end{enumerate}
\end{prop}
\begin{proof}

(1) Assume that $z_1\in\AP(S_1,x)$ and $z_2\in\AP(S_2,q)$. If
$u=qz_1+pz_2\notin\AP(S,qx)$, then $u-qx=qz_1+pz_2-qx=qs_1+ps_2$ for
some $s_1\in S_1$, $s_2\in S_2$. Therefore, $qz_1+pz_2=q(s_1+x)+ps_2$.
Note that $\gcd(p,q)=1$ and $z_1\in\AP(S_1,x)$, so that
$s_1+x=z_1+\alpha p$ for some $\alpha\geq 0$. Hence $pz_2=q\alpha
p+ps_2$ and so $z_2=q\alpha +s_2$. But $z_2\in\AP(S_2,q)$, so that
$\alpha =0$ which implies that $z_1=s_1+x$ and contradicts the
assumption that $z_1\in\AP(S_1,x)$.

We have showed that $\Omega=\{qz_1+pz_2 ; z_1\in\AP(S_1,x),
z_2\in\AP(S_2,q)\}$ is a subset of $\AP(S,qx)$. Now, to prove  the
equality, it is enough to show that $\Omega$ has exactly $qx$
elements.  If $qz_1+pz_2=qs_1+ps_2$ for some $z_1, s_1\in\AP(S_1,x)$
and $z_2, s_2\in\AP(S_2,q)$, then $z_2$ and $s_2$ are congruent modulo
$q$, since $\gcd(p,q)=1$. On the other hand $z_2,s_2\in\AP(S_2,q)$, so
that $z_2=s_2$ and then $z_1=s_1$. Hence
$\#\Omega=\#\AP(S_2,q)\times\#\AP(S_1,x)=qx$.

(2) By part (1), there exist $s_1\in\AP(S_1,x)$ and
$s_2\in\AP(S_2,q)$ such that $qs_1+ps_2=qz_1+pz_2$. Since
$\gcd(p,q)=1$, we have that $z_2$ and $s_2$ are congruent modulo $q$. So that
$z_2=s_2+\alpha q$ for some $\alpha\geq 0$. Hence $qz_1+p\alpha
q=qs_1$ and so $s_1=p\alpha +z_1$. But now it is clear that
$z_1\in\AP(S_1,x)$, because $s_1\in\AP(S_1,x)$.

\end{proof}

Next consequence of the above proposition will be useful in order to determine the Gorenstein property of an specific gluing.

\begin{cor}\label{max}
Let $S=<qm_1,\ldots,qm_d,pn_1,\ldots,pn_k>$ be a gluing of $S_1$ and
$S_2$. Let $x, z_1\in S_1$ and $z_2\in S_2$. The following hold:
\begin{enumerate}
\item If $u=qz_1+pz_2\in\emph\Max\emph\AP(S,qx)$ and $z_2\in \emph\AP(S_2,q)$, then $z_1\in\emph\Max\emph\AP(S_1,x)$ and $z_2\in\emph\Max\emph\AP(S_2,q)$.
\item If $z_1\in\emph{\Max}\emph\AP(S_1,x)$ and $z_2\in\emph{\Max}\emph\AP(S_2,q)$, and there  is only one element in $\emph\Max\emph\AP(S_1,x)$, or only one element in $\emph\Max\emph\AP(S_2,q)$, then $u=qz_1+pz_2\in\emph\Max\emph\AP(S,qx)$.
\end{enumerate}
\end{cor}

\begin{proof}
(1) By Proposition \ref{AP} (2) we have that $z_1\in \AP(S_1,x)$. If $z_1$ is not maximal in $\AP(S_1,x)$ there exists $x_1\neq 0$ such that $z_1+x_1 \in \AP(S_1,x)$, so by Proposition \ref{AP} (1) we have that $u + qx_1 = q(z_1+x_1) + pz_2 \in\AP(S,qx)$, a contradiction. With a similar argument we have that $z_2\in\Max\AP(S_2,q)$

(2) By Proposition \ref{AP} (1), assume that there exist $x_1,t_1 \in \AP(S_1,x), x_2,t_2  \in \AP(S_2,q)$ such that $q(z_1+t_1) + p(z_2+t_2) = qx_1+px_2 \in \Max\AP(S,qx)$. By Proposition \ref{AP} (2), $z_1+t_1\in \AP(S_1,x)$, and by (1), $x_1\in\Max\AP(S_1,x)$ and $x_2 \in \Max\AP(S_2,q)$. Because $z_1$ is maximal we have that $t_1=0$. If we assume that there is only one element in $\Max\AP(S_1,x)$, then $z_1=x_1$ and so $z_2+t_2 = x_2$. Hence because $z_2$ is maximal we get that $t_2=0$. Now assume that there is only one element in $\Max\AP(S_2,q)$. Then $z_2 = x_2$ and $qz_1 + pt_2=qx_1$. Hence $z_1$ and $x_1$ are congruent modulo $(p)$ and so they are equal. Hence $t_2=0$ as well.
\end{proof}

\begin{rem}
Considering the definition of gluing, we may replace $q$ by $p$ in
the above proposition and get the same results for $\AP(S,py)$ where
$y\in S_2$. In other words, $\AP(S,py)=\{qz_1+pz_2;
z_1\in\AP(S_1,p), z_2\in\AP(S_2,y)\}$ and for all elements
$qz_1+pz_2\in\AP(S,py)$, we have $z_2\in\AP(S_2,y)$. In particular there will also be a similar statement as in Corollary \ref{max}.
\end{rem}

Now, we may characterize the symmetric property of the gluing of two numerical semigroups. The fact that the gluing of two symmetric numerical semigroups is symmetric is already well known, see for instance \cite[Proposition 9.11]{GR}, so we only need to prove one of the implications.

\begin{cor}\label{symmetric}
Let  $S=<qm_1,\ldots,qm_d,pn_1,\ldots,pn_k>$ be a gluing of
$S_1$ and $S_2$. Then, the following are equivalent:
\begin{enumerate}
\item $S$ is symmetric, and $S_1$ (or $S_2$) is symmetric.
\item $S_1$ and $S_2$ are symmetric.
\end{enumerate}
\end{cor}
\begin{proof}
(1) $\Rightarrow$ (2) Let $z_1\in \Max\AP(S_1,p)$ and $z_2 \in \Max\AP(S_2,q)$. Then, by Corollary \ref{max} (2), $u=qz_1+pz_2\in \Max\AP(S,qp)$. By hypothesis, there is only one element in $\Max\AP(S,qp)$, and by Proposition \ref{AP}, this element can only be represented in a unique way as an element of the form $qs_1+ps_2$ with $s_1\in \AP(S_1,p)$ and $s_2 \in \AP(S_2,q)$. So $z_1$ is the only element in $\Max\AP(S_1,p)$ and $z_2$ the only one in $\Max\AP(S_2,q)$.

\end{proof}

The following Proposition provides a useful way to present the elements of a specific gluing in a unique way, which plays an essential role in our approach.

\begin{prop}\label{pur}
Let $S=<qm_1,\ldots,qm_d,pn_1,\ldots,pn_k>$ be a specific gluing of
$S_1$ and $S_2$. If $u\in S$, then
\begin{enumerate}
\item  there exist $z_1\in S_1$ and $z_2\in\emph\AP(S_2,q)$ such that $u=qz_1+pz_2$.
\item If $u=qs_1+ps_2$ for some $s_1\in S_1$ and $s_2\in \emph\AP(S_2,q)$, then $s_1=z_1$, $s_2=z_2$ and so  $\emph\ord_S(u)=\emph\ord_{S_1}(s_1)+\emph\ord_{S_2}(s_2)$.
\end{enumerate}
\end{prop}
\begin{proof}
(1) By Lemma \ref{ex}, there exist $s_1\in S_1$ and $s_2\in S_2$
such that $u=qs_1+ps_2$ and
$\ord_S(u)=\ord_{S_1}(s_1)+\ord_{S_2}(s_2)$. Among all $s_1\in
S_1$ with this property, we choose $z_1$ with the maximum possible
order. That is, $u=qz_1+pz_2$ and
$\ord_{S_1}(z_1)=\max\{\ord_{S_1}(s_1) ; u=qs_1+ps_2 \mbox{ for some
} s_2\in S_2, \ord_S(u)=\ord_{S_1}(s_1)+\ord_{S_2}(s_2)\}$.

We are going to prove that $z_2\in\AP(S_2,q)$. If not, then $z_2=s_2+q$ for some $s_2\in S_2$.
Note that by Remark \ref{L} $\ord_{S_2}(z_2)\leq\ord_{S_2}(s_2)+\ord_{S_2}(q)+l_q(S_2)$.
On the other hand $u=qz_1+p(s_2+q)=q(z_1+p)+ps_2$. Thus
\[\begin{array}{ll}
\ord_{S_1}(z_1)+\ord_{S_2}(z_2)&=\ord_S(u)\\
&\geq\ord_{S_1}(z_1+p)+\ord_{S_2}(s_2)\\
&\geq\ord_{S_1}(z_1+p)+\ord_{S_2}(z_2)-\ord_{S_2}(q)-l_q(S_2)\\
&\geq\ord_{S_1}(z_1)+\ord_{S_1}(p)+\ord_{S_2}(z_2)-\ord_{S_2}(q)-l_q(S_2).
\end{array}\]
Hence $\ord_{S_1}(p)-\ord_{S_2}(q)-l_q(S_2)\leq 0$, which implies that
$\ord_{S_1}(p)=\ord_{S_2}(q)+l_q(S_2)$ by the definition of specific gluing, Definition
\ref{sg}. Hence $\ord_{S_2}(u)=\ord_{S_1}(z_1+p)+\ord_{S_2}(s_2)$ which
contradicts our selection of $z_1$.

(2) By part (1), there exist $z_1\in S_1$ and $z_2\in\AP(S_2,q)$ such that $u=qs_1+ps_2=qz_1+pz_2$ and $\ord_S(u)=\ord_{S_1}(z_1)+\ord_{S_2}(z_2)$.
Since  $\gcd(p,q)=1$, we have $z_2$ and $s_2$ are congruent modulo $q$. Hence $z_2=s_2$, because $s_2,z_2\in\AP(S_2,q)$, and so $s_1=z_1$ as well.
\end{proof}

In particular, by taking $s_2=0$ in the above proposition we get:

\begin{cor}\label{order}
Let  $S=<qm_1,\ldots,qm_d,pn_1,\ldots,pn_k>$ be a specific gluing of
$S_1$ and $S_2$. Then
 $\ord_{S}(qx)= \ord_{S_1}(x)$ for all  $x\in S_1$.
\end{cor}

And we also get the concrete value of the multiplicity of a specific gluing:

\begin{cor}\label{multiplicity}
Let  $S=<qm_1,\ldots,qm_d,pn_1,\ldots,pn_k>$ be a specific gluing of
$S_1$ and $S_2$. Then $m(S)= qm_1$.
\end{cor}
\begin{proof}
By Remark \ref{m} we only have to prove that $qm_1\leq pn_1$. Assume the contrary. Then $pn_1<qm_1$ and the multiplicity of $S$ is equal to $pn_1$. Hence by Lemma \ref{zero} (1), $\ord_S(pn_1qm_1)=qm_1$. But by Corollary \ref{order}, $\ord_S(pn_1qm_1) = \ord_{S_1}(pn_1m_1) = pn_1$, again by Lemma \ref{zero} (1), which is a contradiction.
\end{proof}

Our next proposition partially describes the torsion of the tangent cone of $S$ in terms of the torsion of the tangent cone of $S_1$. Namely:

\begin{prop}\label{t}
Let  $S=<qm_1,\ldots,qm_d,pn_1,\ldots,pn_k>$ be a specific gluing of
$S_1$ and $S_2$.
 If $\ord_S(u+qw)>\ord_S(u)+\ord_S(qw)$ for some
$u\in S$ and $w\in S_1$, then $u=qz_1+pz_2$ for some $z_1\in S_1$
and $z_2\in S_2$ such that
$\ord_S(u)=\ord_{S_1}(z_1)+\ord_{S_2}(z_2)$ and
$\ord_{S_1}(z_1+w)>\ord_{S_1}(z_1)+\ord_{S_1}(w)$.
\end{prop}
\begin{proof}
By Proposition \ref{pur}, there exist $z_1, s_1\in S_1$ and $z_2,s_2\in \AP(S_2,q)$ such that
$u=qz_1+pz_2$, $u+qw=qs_1+ps_2$,
$\ord_S(u)=\ord_{S_1}(z_1)+\ord_{S_2}(z_2)$ and
$\ord_S(u+qw)=\ord_{S_1}(s_1)+\ord_{S_2}(s_2)$. Hence $q(z_1+w)+pz_2=qs_1+ps_2$, and since $\gcd(p,q)=1$ we get that $z_2$ and $s_2$ are congruent modulo $q$. Because $z_2, s_2\in\AP(S_2,q)$ we must have that $z_2=s_2$ and so $z_1+w=s_1$. By
assumption, $\ord_S(u+qw)>\ord_S(u)+\ord_S(qw)$ and by Corollary
\ref{order}, $\ord_S(qw)=\ord_{S_1}(w)$. So we have
$$\begin{array}{ll}
\ord_{S_1}(z_1+w)+\ord_{S_2}(z_2)&=\ord_{S_1}(s_1)+\ord_{S_2}(s_2)\\&>\ord_{S_1}(z_1)+\ord_{S_2}(z_2)+\ord_{S}(qw)\\
&=\ord_{S_1}(z_1)+\ord_{S_2}(s_2)+\ord_{S_1}(w).
\end{array}$$
Thus $\ord_{S_1}(z_1+w)>\ord_{S_1}(z_1)+\ord_{S_1}(w)$ and the
result follows.
\end{proof}

As a consequence of the above proposition, we
obtain that  the Cohen-Macaulay property of the tangent cone of a
specific gluing is equivalent to the Cohen-Macaulay property of
$G(S_1)$, which is one of the main results of this paper.

\begin{thm}\label{nc}
Let  $S=<qm_1,\ldots,qm_d,pn_1,\ldots,pn_k>$ be a  specific gluing
of $S_1$ and $S_2$. Then $G(S)$ is Cohen-Macaulay if and only if
$G(S_1)$ is Cohen-Macaulay.
\end{thm}
\begin{proof}
Observe first that by Corollary \ref{multiplicity}, $m(S)=qm_1$. Assume that G(S) is not Cohen-Macaulay. Then, by Lemma \ref{CM}, $\ord_S(u+qm_1)>\ord_s(u)+1$ for some $u\in S$. Hence by Proposition \ref{t}, there exists $z_1\in S_1$ such that $\ord_{S_1}(z_1+m_1)>\ord_{S_1}(z_1)+1$, and so again by Lemma \ref{CM} $G(S_1)$ is not Cohen-Macaulay.

Assume now that $G(S_1)$ is not Cohen-Macaulay. By Lemma \ref{CM} there exists $u\in S_1$ such that $\ord_{S_1}(u+m_1) > \ord_{S_1}(u) +1$. Hence by Corollary \ref{order}, $\ord_S(qu +qm_1) = \ord_{S_1}(u+m_1) > \ord_{S_1}(u) +1 = \ord_S(qu) +1$, and by Lemma \ref{CM} $G(S)$ is not Cohen-Macaulay.
\end{proof}

Now, the following corollary completes \cite[Theorem 2.6]{AMS}, where only the necessary part was obtained by a different method.

\begin{cor}\label{niceCM}
Let  $S=<qm_1,\ldots,qm_d,pn_1,\ldots,pn_k>$ be a nice  gluing of
$S_1$ and $S_2$. If $G(S_2)$ is Cohen-Macaulay, then $G(S)$ is
Cohen-Macaulay if and only if $G(S_1)$ is Cohen-Macaulay.
\end{cor}
\begin{proof}
Just note that $S$ is now a specific gluing of $S_1$ and $S_2$, as we
have seen in Remark \ref{ns}. Hence the result follows from Theorem
\ref{nc}.
\end{proof}

Next, we deal with the Gorenstein property of the tangent cone of an specific gluing. First, we must characterize the $M$-purity.

\begin{prop}\label{M-purity}
Let $S=<qm_1,\ldots,qm_d,pn_1,\ldots,pn_k>$ be a specific gluing of
$S_1$ and $S_2$. For any $x\in S_1$, the following are equivalent:
\begin{enumerate}
\item $S$ is symmetric and $M$-pure with respect to $qx$, and $S_1$ (or $S_2$) is symmetric.
\item $S_1$ is symmetric and $M$-pure with respect to $x$ and $S_2$ is symmetric and $M$-pure with respect to $q$.
\end{enumerate}
\end{prop}
\begin{proof}
Observe first that by Corollary \ref{symmetric}, $S$, $S_1$, and $S_2$ are symmetric.

(2) $\Rightarrow$ (1). By hypothesis and by Corollary \ref{max}, we have that $\Max\AP(S,qm_1)=\{w=qw_1+pw_2\}$, where
$\Max\AP(S_1,m_1)=\Max_M\AP(S_1,m_1)=\{w_1\}$ and
$\Max\AP(S_2,q)=\Max_M\AP(S_2,q)=\{w_2\}$.
Now it is enough to show that, in fact,
$\Max_M\AP(S,qm_1)=\{w\}$. For this, we will see that $u\preceq_Mw$ for all
$u\in\AP(S,qm_1)$. So let $u\in\AP(S,qm_1)$. Then, $u=qz_1+pz_2$ for some
$z_1\in\AP(S_1,m_1)$, $z_2\in\AP(S_2,q)$ by Proposition \ref{AP}. Hence
$\ord_S(u)=\ord_{S_1}(z_1)+\ord_{S_2}(z_2)$, by Proposition \ref{pur}.
On the other hand $z_1\preceq_Mw_1$ and $z_2\preceq_Mw_2$, that is
$z_1+s_1=w_1$ with $\ord_{S_1}(z_1)+\ord_{S_1}(s_1)=\ord_{S_1}(w_1)$
and $z_2+s_2=w_2$ with
$\ord_{S_2}(z_2)+\ord_{S_2}(s_2)=\ord_{S_2}(w_2)$, for some $s_1\in \AP(S_1,m_1)$ and $s_2\in \AP(S_2,q)$. Let
$u'=qs_1+ps_2$. Note that
$\ord_S(u')=\ord_{S_1}(s_1)+\ord_{S_2}(s_2)$ and also
$\ord_{S}(w)=\ord_{S_1}(w_1)+\ord_{S_2}(w_2)$, by Proposition \ref{pur}. Hence we have $u+u'=w$ and
$$\begin{array}{ll}
\ord_S(w)&=\ord_{S_1}(w_1)+\ord_{S_2}(w_2)\\
&=\ord_{S_1}(z_1)+\ord_{S_1}(s_1)+\ord_{S_2}(z_2)+\ord_{S_2}(s_2)\\
&=(\ord_{S_1}(z_1)+\ord_{S_2}(z_2))+(\ord_{S_1}(s_1)+\ord_{S_2}(s_2))\\
&= \ord_{S_1}(u)+\ord_{S_2}(u'),\end{array}$$ which means that
$u\preceq_Mw$, as we wanted to see.

(1) $\Rightarrow$ (2) By hypothesis we have that $\Max\AP(S,qm_1)=\Max_M\AP(S,qm_1)=\{w\}$. On the other hand, by Corollary \ref{max},
$w=qz_1+pz_2$, where  $\max\AP(S_1,m_1)=\{z_1\}$ and
$\max\AP(S_2,q)=\{z_2\}$. Thus it is enough to show that
$\Max_M\AP(S_1,m_1)=\{z_1\}$ and $\Max_M\AP(S_2,q)=\{z_2\}$. Let
$x\in\AP(S_1,m_1)$. Then $qx\in\AP(S,qm_1)$, by Proposition
\ref{AP}. As a consequence, $qx\preceq_M w$, that is $qx+qs_1+ps_2=qz_1+pz_2$
for some $s_1\in\AP(S_1,m_1)$ and $s_2\in\AP(S_2,q)$, with
$\ord_S(w)=\ord_S(qx)+\ord_S(qs_1+ps_2)$. Because $\gcd(p,q)=1$ we get that
$z_2$ is congruent with $s_2$ modulo $q$. Now, since $z_2,s_2\in\AP(S_2,q)$ we have that
$z_2=s_2$ and so $x+s_1=z_1$.  On the other hand
$\ord_S(qs_1+ps_2)=\ord_{S_1}(s_1)+\ord_{S_2}(s_2)$,
$\ord_S(w)=\ord_{S_1}(z_1)+\ord_{S_2}(z_2)$ and
$\ord_{S}(qx)=\ord_{S_1}(x)$, by Lemma \ref{pur}. Hence
$\ord_{S_1}(x)+\ord_{S_1}(s_1)=\ord_{S_1}(z_1)$ and so $x\preceq_M
z_1$. Thus $\Max_M\AP(S_1,m_1)=\{z_1\}$ as we wanted to see.

By a similar argument we obtain that $\Max_M\AP(S_2,q)=\{z_2\}$ and the result follows.
\end{proof}

Now, we may state the following characterization of the Gorenstein property of a specific gluing, which is another of the main results of this paper.
\begin{thm}\label{Gor}
Let $S=<qm_1,\ldots,qm_d,pn_1,\ldots,pn_k>$ be a specific gluing of
$S_1$ and $S_2$. Assume that $S_2$ is symmetric and $M$-pure with respect to $q$. Then $G(S)$ is Gorenstein if and only if $G(S_1)$ is Gorenstein.
\end{thm}
\begin{proof}
Assume that $G(S)$ is Gorenstein. Then, by Theorem \ref{nc} $G(S_1)$ is Cohen-Macaulay. On the other hand, by Proposition \ref{PS} $S$ is symmetric and $M$-pure with respect to $qm_1$. Hence by Proposition \ref{M-purity} $S_1$ is symmetric and $M$-pure with respect to $m_1$. Now, again by Proposition \ref{PS} we have that $G(S_1)$ is Gorenstein.

Assume that $G(S_1)$ is Gorenstein. By Theorem \ref{nc} $G(S)$ is Cohen-Macaulay. Also, by Proposition \ref{PS} $S_1$ is symmetric and $M$-pure with respect to $m_1$. Hence by hypothesis and Proposition \ref{M-purity}, $S$ is symmetric and $M$-pure with respect to $qm_1$. Finally, once more by Proposition \ref{PS} we have that $G(S)$ is Gorenstein.
\end{proof}

In particular, by Corollary \ref{niceCM} we have the following concrete result in the case of a nice gluing.

\begin{cor}\label{niceGor}
Let  $S=<qm_1,\ldots,qm_d,pn_1,\ldots,pn_k>$ be a nice  gluing of
$S_1$ and $S_2$. If $G(S_2)$ is Gorenstein, then $G(S)$ is
Gorentein if and only if $G(S_1)$ is Gorenstein.
\end{cor}

We finish this section by observing that as a natural generalization of the Cohen-Macaulay property, one may consider the
Buchsbaum state. The following example shows that this condition is not preserved by nice gluing, even if $G(S_2)$ is Cohen-Macaulay.

\begin{ex}  Let $S_1=<4,11,29>$ and $S_2=<2,3>$. Then $G(S_1)$ is Buchsbaum, since $\lambda(\H^0_\fm(G(S))=1$ by \cite[Example 4.5(3)]{CZ1}.
Now consider the numerical semigroup $S=<16,44,116,30,45>$ obtained as a nice gluing of $S_1$ and
$S_2$ with $q=2\times 2=4$, $p=4+11=15$. Then
$\ord_S(116)=\ord_S(16)=1$ and $132=116+16=3\times 44\in 3M$. So
that $(t^{116})^*(t^{16})^*=0$ which means that
$(t^{116})^*\in\H^0_\fm(G(S)$. On the other hand $\ord_S(116+30)=2$
which means that  $(t^{116})^*(t^{30})^*\neq 0$. Hence $G(S)$ is not
Buchsbaum.

\end{ex}

\section{non--decreasing Hilbert functions }

In this section we study the growth of the Hilbert functions of
monomial curves obtained by gluing. Recall that for a numerical
semigroup $S$, the Hilbert function of $S$ is the function given for
any non-negative integer $n$ as $H(n) = \lambda (\fm^n/\fm^{n+1}) =
\# (nM \setminus (n+1)M)$. Our goal is to construct monomial curves
with non--decreasing Hilbert functions by gluing monomial curves
with the same property. For instance, as we have observed in Section
1, the Hilbert function of a monomial curve with Cohen-Macaulay
tangent cone is non--decreasing. Hence by Theorem \ref{nc}, we know
that if $S$ is a specific gluing of two numerical semigroups $S_1$
and $S_2$ such that $S_1$ has a Cohen-Macaulay tangent cone, then
$S$ has a Cohen-Macaulay tangent cone and so its Hilbert function is
non--decreasing. The following result shows that in fact we can
replace the Cohen-Macaulay property of the tangent cone of $S_1$ by
the non--decreasing property of its Hilbert function to get the same
property for $S$.

\begin{thm}\label{Hn}
Let  $S=<qm_1,\ldots,qm_d,pn_1,\ldots,pn_k>$ be a specific  gluing
of $S_1$ and $S_2$. If $S_1$ has non--decreasing Hilbert function,
 then $G(S)$ has non--decreasing Hilbert function.
\end{thm}
\begin{proof}
Let $n\geq 0$ and $u\in nM\setminus(n+1)M$. By Lemma \ref{pur}, there exists a representation $u=qz_1+pz_2$ such that $z_2\in\AP(S_2,q)$ and  $\ord_S(u)=\ord_{S_1}(z_1)+\ord_{S_2}(z_2)$. Let $k=\ord_{S_1}(z_1)$. Since the Hilbert function of $S_1$ is
non--decreasing, there exists an injective map $f_k:kM_1\setminus(k+1)M_1\hookrightarrow
(k+1)M_1\setminus(k+2)M_1$. Now, taking into account Lemma \ref{pur}, we have that
$$\begin{array}{ll}
\ord_S(qf_k(z_1)+pz_2)&=\ord_{S_1}(f_k(z_1))+\ord_{S_2}(z_2)\\
&=k+1+\ord_{S_2}(z_2)=\ord_S(u)+1.\end{array}$$ Hence the function
$$g_n:nM\setminus(n+1)M\rightarrow (n+1)M\setminus(n+2)M$$ defined
by $g_n(u):=qf_k(z_1)+pz_2$ is well defined. In order to show that
the Hilbert function of $S$ is non--decreasing, it is enough to see
that $g_n$ is an injection. Assume that $u'=qz'_1+pz'_2\in
nM\setminus(n+1)M$ such that $g_n(u)=g_n(u')$. By Lemma \ref{pur} we
may assume that $z'_2 \in \AP(S_2,q)$. Thus
$qf_k(z_1)+pz_2=qf_{k'}(z'_1)+pz'_2$, where $k'=\ord_{S_1}(z_1)$.
Since $\gcd(p,q)=1$, we have $z_2$ is congruent with $z'_2$ modulo
$q$. On the other hand  $z_2, z'_2\in\AP(S_2,q)$, hence $z_2=z'_2$
and so $z_1=z'_1$, and $u=u'$ as well.
\end{proof}

As we have observed in  Remark \ref{ns}, if
$S$ is a nice gluing of $S_1$
and $S_2$, where $G(S_2)$ is Cohen-Macaulay, then $S$ is a specific
gluing of $S_1$ and $S_2$. Hence we may state the following result which
is also proved in \cite[Theorem 3.1]{AMS} by a different technique.

\begin{cor}\label{ngh}
Let $S=<qm_1,\ldots,qm_d,pn_1,\ldots,pn_k>$ be a nice  gluing of
$S_1$ and $S_2$. If $G(S_2)$ is Cohen-Macaulay and $S_1$ has
non--decreasing Hilbert function,
 then $G(S)$ has non--decreasing Hilbert function.
\end{cor}

\begin{cor}
Let $S=<qm_1,\ldots,qm_d,pn_1,\ldots,pn_k>$ be a gluing of $S_1$ and
$S_2$. Assume that $S_1$ has non--decreasing Hilbert function. Then,
for a given $q$, all numerical semigroups $S$ obtained by gluing of
$S_1$ and $S_2$, except finitely many of them, have non--decreasing
Hilbert functions.
\end{cor}

\begin{proof}
It follows by Theorem \ref{Hn} and Remark \ref{fs}.
\end{proof}

It is well known that a numerical semigroup $S$ of embedding dimension $n>1$ is free if and only if $S$ is a gluing of a  free numerical semigroup
of embedding dimension $n-1$ and $\N$, see \cite[Theorem 8.16]{GR}. Hence $S$ may be described as $S = <qm_1, \dots , qm_d, p>$ where $S_1 = <m_1, \dots, m_d>$ is a free numerical semigroup of multiplicity $d$, $p\in S_1 \setminus \{m_1, \dots m_d\}$, $q>1$  and $\gcd(p,q)=1$. In particular, free numerical semigroups are complete intersection, see \cite[Corollary 9.7]{GR}. Observe that a free numerical semigroup may be obtained as consecutive gluings of free numerical semigroups and $\N$, starting from $\N$ itself. If each of these consecutive gluings is nice (equivalently specific under these hypothesis), then we get from Corollary \ref{niceCM} and by induction that such a numerical semigroup has a Cohen-Macaulay tangent cone and so non-decreasing Hilbert function.

Now, we apply our techniques on this kind of gluing, to get more families of Gorenstein monomial curves with non--decreasing Hilbert functions.
First recall the following definition from \cite{AM}.
\begin{defn}
 The semigroup $S=<qm_1,\ldots,qm_d,p>$, obtained
by gluing of $S_1$ and $\N$, is called an extension of $S_1$. If it
is a nice gluing, then we call $S$ a nice extension. In other words
$S$ is an extension of $S_1$ when $q>1$, $p\in
S_1\setminus\{m_1,\ldots,m_d\}$ and $\gcd(p,q)=1$.  Moreover, if
$q\leq \ord_{S_1}(p)$, then $S$ is called a nice extension of $S_1$.
\end{defn}

As a special case of nice gluing, we know by Corollary \ref{niceCM}, that if $G(S_1)$ is Cohen-Macaulay, then any nice extension of $S_1$ has
Cohen-Macaulay tangent cone. The following example shows that we can
not remove the nice condition even for extensions.

\begin{ex}
Let $S_1=<2,5>$, $p=2+5=7$ and $q=3$. Then $S=<6, 15, 7>$ is an
extension of $S_1$ which is not nice. Now $15+6=21=3\times7$ so that
$\ord_S(15+6)>2=\ord_{S_1}(15)+\ord_{S_1}(6)$. Hence $G(S)$ is not
Cohen-Macaulay by Corollary \ref{CM}.
\end{ex}

The following result provides a large family of extensions whose tangent cones are always Cohen-Macaulay, independently from
the tangent cone of $S_1$. The condition we impose is somehow complementary to being a nice extension.

\begin{thm}\label{ext}
Let $S=<qm_1,\ldots,qm_d,p>$ be an extension of $S_1$. If $p<q$,
then  $G(S)$ is Cohen-Macaulay and so the Hilbert function of $S$ is
non--decreasing.
\end{thm}
\begin{proof}
Note that $m(S)=p$ and so by using Corollary \ref{CMl} it is enough to show that $\ord_S(s+p)=\ord_S(s)+1$ for all $s\in S$ with $\ord_S(s)\leq p-1$.
Let $s=qz_1+pz_2$ for some $z_1\in S_1$ and $z_2\in\N$ such that $\ord_S(s)=\ord_{S_1}(z_1)+z_2\leq p-1$ (that exists by Lemma \ref{ex}). There also exist $s_1\in S_1$ and $s_2\in\N$ such that
\begin{equation}
s+p=qz_1+p(z_2+1)=qs_1+ps_2
\end{equation}
and  $\ord_S(s+p)=\ord_{S_1}(s_1)+s_2$. If $s_2>z_2+1$, then $qz_1=qs_1+\alpha p$, where $\alpha=s_2-z_2-1>0$. Now, since $\gcd(q,p)=1$ we have that $\alpha = q\alpha'$ with $\alpha'>0$. Thus
$\ord_S(qz_1)\geq \ord_S(\alpha'pq) = \alpha' q \geq q >p$, which is a contradiction. Hence $z_2+1-s_2\geq 0$ and we may write $qz_1+p(z_2+1-s_2)=qs_1$. Again because $\gcd(q,p)=1$ we have that $z_2+1-s_2=\beta q$,
for some $\beta\geq 0$. If $\beta \neq 0$ then $0<z_2+1-s_2\leq p-s_2<q-s_2$ and so $\beta q<q-s_2$. Hence $\beta=0$ and we get
$$\begin{array}{ll}
\ord_S(s+p)&=\ord_{S_1}(s_1)+s_2\\
&=\ord_{S_1}(z_1)+z_2+1=\ord_S(s)+1.\end{array}$$
as we wanted to show.
\end{proof}

\begin{ex}
Let $S_1=<5,6,13>$. Then $G(S_1)$ is not Cohen-Macaulay by
\cite[Example 2.4]{CZ1}. Now consider $S=<11,60,68,156>$ as an
extension of $S_1$ by $p=11$ and $q=12$. Note that
$\ord_{S_1}(p)=2<q$ and so $S$ is not a nice extension of $S_1$.
Using the NumericalSgps package of GAP \cite{Num} we obtain the
following Apéry table of $S$.
\[
\begin{array}{|c|c|c|c|c|c|c|c|c|c|c|c|}\hline
 \AP(S)& 11 & 60 & 72& 120& 144& 156& 180& 216& 228& 240& 300\\
\hline
 \AP(M)& 11 & 60 & 72& 120& 144& 156& 180& 216& 228& 240& 300\\
\hline
\AP(2M)& 22&  71&  83& 120& 144& 167& 180& 216& 228& 240& 300\\
 \hline
\AP(3M)& 33& 82& 94&  131&  155& 178& 180& 216& 239& 240& 300\\
\hline
\AP(4M)& 44& 93& 105& 142& 166&  189& 191& 227&  250& 240& 300\\
\hline
\AP(5M) & 55& 104& 116& 153& 177& 200& 202& 238& 261& 251& 300\\
\hline  \end{array} \] Since there is no true landing in the Apéry
table of $S$, we obtain that $G(S)$ is Cohen-Macaulay by Remark
\ref{lr} and Corollary \ref{CMl}.
\end{ex}

The assumption $p<q$ in  Theorem \ref{ext} implies that $p<qm_1$ and so $p = m(S) < qm_1$. But this condition is not enough, as it may happen that $q<m(S)=p<qm_1$ and the following example shows that,
in this case, the tangent cone of $S$ is not necessarily Cohen-Macaulay, even if the tangent cone of $S_1$ is Cohen-Macaulay.

\begin{ex}
Let $S=<5, 8, 28>$, then $S$ is an extension of $S_1=<2,7>$ with $q=4$ and $p=5$. Note that $28+4\times 5=6\times 8$ and so $\ord_S(28+4\times 5)>1+4$. Hence
$G(S)$ is not Cohen-Macaulay by Corollary \ref{CM}.
\end{ex}

The following corollary provides an easy way to get numerical
semigroups with non-decreasing Hilbert functions.

\begin{cor}
Let $S=<a_1,\ldots,a_n>$ be a numerical semigroup with embedding dimension $n>1$. If $\gcd(a_2,\ldots,a_n)>a_1$, then
$G(S)$ is Cohen-Macaulay and so the Hilbert function of $S$ is
non--decreasing.
\end{cor}

\begin{proof}
Let $d=\gcd(a_2,\ldots,a_n)$, then $S$ is an extension of $S_1=<\frac{a_2}{d},\ldots,\frac{a_n}{d}>$ with $p=a_1$ and $q=d$. Now the result follows by Theorem \ref{ext}.
\end{proof}

As a consequence of Corollary \ref{ngh} and Theorem \ref{ext} we have the following result, which illustrates that
  for a numerical semigroup $S_1$ with non--decreasing Hilbert function and an integer $q>1$, all extensions of $S_1$ by $q$, except finitely many of them which in fact are bounded by a strong numerical condition, have non--decreasing Hilbert functions.

\begin{thm}
Let $S=<qm_1,\ldots,qm_d,p>$ be an extension of $S_1$. Assume that $S_1$ has a non--decreasing Hilbert function.
 If the Hilbert function of $S$ is decreasing, then $\ord_{S_1}(p)<q<p$.
\end{thm}

\end{document}